\documentclass[11pt]{article}
\usepackage{amsmath, amssymb, amsthm}
\usepackage{verbatim}
\usepackage{multicol}
\usepackage{enumerate}
\usepackage{comment}
\usepackage[none]{hyphenat}
\usepackage{hyperref}
\hypersetup{
	colorlinks=true,
	linkcolor=blue,
	filecolor=magenta,
	urlcolor=cyan,
	citecolor=blue
}

\usepackage{pgf}
\usepackage{tikz}
\usetikzlibrary{math}
\usetikzlibrary{positioning,arrows,shapes,decorations.markings,decorations.pathreplacing,matrix,patterns}
\tikzstyle{vertex}=[circle,draw=black,fill=black,inner sep=0,minimum size=3pt,text=white,font=\footnotesize]
\usepackage{cleveref}

\date{}
\title{\vspace{-0.8cm}Colorings with only rainbow arithmetic progressions}
\author{
	J\'{a}nos Pach \thanks{R\'enyi Institute, Budapest. \emph{e-mail}: \textbf{pach@cims.nyu.edu}} \thanks{\'{E}cole Polytechnique F\'{e}d\'{e}rale de Lausanne. Research partially supported by Swiss National Science Foundation grants no. 200020-162884 and 200021-175977. }
	\and
	Istv\'{a}n Tomon \thanks{ETH Zurich. \emph{e-mail}: \textbf{istvan.tomon@math.ethz.ch}, Research supported by SNSF grant 200021-149111.}
}

\theoremstyle{plain}
\newtheorem{theorem}{Theorem}

\newtheorem{corollary}[theorem]{Corollary}
\newtheorem{claim}[theorem]{Claim}

\theoremstyle{definition}

\begin{document}
\maketitle

\begin{abstract}
If we want to color $1,2,\ldots,n$ with the property that all 3-term arithmetic progressions are {\em rainbow} (that is, their elements receive 3 distinct colors), then, obviously, we need to use at least $n/2$ colors. Surprisingly, much fewer colors suffice if we are allowed to leave a negligible proportion of integers uncolored. Specifically, we prove that there exist $\alpha,\beta<1$ such that for every $n$, there is a subset $A$ of $\{1,2,\ldots,n\}$ of size at least $n-n^{\alpha}$, the elements of which can be colored with $n^{\beta}$ colors with the property that every 3-term arithmetic progression in $A$ is rainbow. Moreover, $\beta$ can be chosen to be arbitrarily small. Our result can be easily extended to $k$-term arithmetic progressions for any $k\ge 3$.

As a corollary, we obtain the following result of Alon, Moitra, and Sudakov, which can be used to design efficient communication protocols over shared directional multi-channels. There exist $\alpha',\beta'<1$ such that for every $n$, there is a graph with $n$ vertices and at least $\binom{n}{2}-n^{1+\alpha'}$ edges, whose edge set can be partitioned into at most $n^{1+\beta'}$ {\em induced matchings}.
\end{abstract}

\begin{center}
	{\it Dedicated to the 80th birthday of Endre Szemer\'edi.}
\end{center}

\section{Introduction}
Szemer\'edi's regularity lemma \cite{Sz78} started a new chapter in extremal combinatorics and in additive number theory. In particular, it was instrumental in proving a famous conjecture of Erd\H os and Tur\'an, according to which, for every real number $\delta > 0$ and every integer $k>0$, there exists a positive integer $n=n(\delta,k)$ such that every subset of $[n]=\{1, 2,\ldots, n\}$ that has at least $\delta n$ elements contains an arithmetic progression of length $k$ (in short, a {\em $k$-AP}); see \cite{Sz75}. The $k=3$ special case of this theorem, originally proved by Roth~\cite{R52}, is equivalent to the {\em triangle removal lemma} of Ruzsa and Szemer\'edi~\cite{RSz78}, which is another direct consequence of the regularity lemma. It has several other equivalent formulations:

\begin{enumerate}
\item If $A$ is subset of $[n]$ with no $3$-AP, then $|A|=o(n)$.
\item If $G$ is a graph on $n$ vertices whose edge set can be partitioned into $n$ \emph{induced} matchings, then $|E(G)|=o(n^{2})$.

\item If $G$ is a graph on $n$ vertices which has $o(n^3)$ triangles, then one can eliminate all triangles by removing $o(n^2)$ edges of $G$.

\item If $H$ is a system of triples of $[n]$ such that every $6$-element subset of $[n]$ contains at most $2$ triples in $H$, then $|H|=o(n^2)$.
\end{enumerate}

\noindent More precisely, the above statements apply to any infinite series of sets $A$, graphs $G$, and triple systems $H$, resp., where $n\rightarrow\infty$.
\smallskip

An old construction of Behrend \cite{B46} shows that there are 3-AP-free sets $A\subset [n]$ of size at least $ne^{-O(\sqrt{\log n})}$, so that 1 is not far from being tight. Ruzsa and Szemer\'edi observed, that Behrend's construction can be used to show the existence of graphs $G$ with $n$ vertices and $|E(G)|\ge n^2e^{-O(\sqrt{\log n})}$ edges that can be partitioned into $n$ induced matchings. Hence, 2 is also nearly tight, and the same is true for 3 and 4.
\smallskip

Szemer\'edi's theorem on arithmetic progressions immediately implies van der Waerden's theorem \cite{vdW27}: For any integer $k\ge 3$, let $c_k(n)$ denote the minimum number of colors needed to color all elements of $[n]$ without creating a {\em monochromatic} $k$-AP. Then we have $\lim_{n\rightarrow\infty}c_k(n)=\infty$.
\smallskip

How many colors do we need if, instead of trying to avoid monochromatic $k$-term arithmetic progressions, we want to make sure that every $k$-term arithmetic progression is {\em rainbow}, that is, all of its elements receive distinct colors? For instance, it is easy to see that for $k=3$, we need at least $n/2$ colors. Surprisingly, it turns out that much fewer colors suffice if we do not insist on coloring {\em all} elements of $[n]$. In particular, there is a subset of $A\subset [n]$ with $|A|=(1-o(1))n$ whose elements can be colored by $n^{o(1)}$ colors with the property that all 3-term arithmetic progressions in $A$ are rainbow.

More precisely, we prove the following result.

\begin{theorem}\label{thm:mainthm}
	There exist $\alpha,\beta<1$ with the following property. For every sufficiently large positive integer $n$, there are a set $A\subset [n]$ with $|A|\geq n-n^{\alpha}$ and a coloring of $A$ with at most $n^{\beta}$ colors such that every 3-term arithmetic progression in $A$ is rainbow.

Moreover, for every $\beta>0$, we can choose $\alpha<1$ satisfying the above conditions.
\end{theorem}

Theorem~\ref{thm:mainthm} can be used to construct graphs with $n$ vertices and $(1-o(1)){n\choose 2}$ edges which can be partitioned into a small number of induced matchings. The first such constructions were found by Alon, Moitra, and Sudakov \cite{AMS13}.

\begin{corollary}\label{cor:maincor}
  There exist $\alpha',\beta'<1$ with the following property. For every sufficiently large positive integer $n$, there is a graph with $n$ vertices and at least $\binom{n}{2}-n^{1+\alpha'}$ edges that can be partitioned into $n^{1+\beta'}$ induced matchings.

Moreover, for every $\beta'>0$, we can choose $\alpha'<1$ satisfying the above condition.
\end{corollary}

Dense graphs that can be partitioned into few induced matchings have been extensively studied, partially due to their applications in graph testing \cite{A02,AS03,AS04,HW03} and testing monotonicity in posets \cite{FNRRS02}. The graphs satisfying the conditions in Corollary \ref{cor:maincor} can be used to design efficient communication protocols over shared directional multi-channels \cite{BLM,AMS13}. Some other interesting graphs decomposable into large matchings were constructed and studied in \cite{FHS17}.
\smallskip

Our proof of Theorem \ref{thm:mainthm} is inspired by the construction of Behrend \cite{B46}, but it also has a lot in common with one of the constructions of Alon, Moitra, and Sudakov \cite{AMS13}. Roughly, the idea of Behrend is to identify the elements of $[n]$ with a high dimensional grid $[C]^{d}$, in which we find a sphere passing through many grid points. These points will correspond to a dense 3-AP-free set in $[n]$. We proceed similarly, but instead of taking a sphere, we take a small neighborhood $S$ of a sphere. If we choose the radii properly, it follows by standard concentration laws that almost all points of the grid $[C]^{d}$ are contained in $S$. On the other hand if 3 points form a 3-AP in $S$, then they must be close to each other. This observation can be explored to give a coloring of $S\cap [C]^{d}$ with the desired properties.
\smallskip

In Section 2 and 3, we prove Theorem~\ref{thm:mainthm} and Corollary~\ref{cor:maincor}, respectively. In the last section, we indicate how to extend Theorem~\ref{thm:mainthm} to $k$-term arithmetic progressions for any $k\ge 3$; see Theorem~\ref{general}.

\section{Rainbow 3-AP's---Proof of Theorem \ref{thm:mainthm}}

We start by setting a few parameters. Let $C$ be a sufficiently large integer. Suppose for simplicity that $n=C^d$ for some integer $d$. The general case can be treated in a similar manner. In the sequel, $\log$ will stand for the natural base logarithm.

Set $\epsilon=\frac{1}{C^3}$ and let $B=\{0,1,\dots,C-1\}^d$, so that $|B|=C^d$. For any $\mathbf{x}\in B$, let $\mathbf{x}(i)\in\{0,1,\ldots,C-1\}$ denote the $i$th coordinate of $\mathbf{x}$, where $1\le i\le d$. Clearly, the map $\phi: B\rightarrow [n]$ defined as $$\phi(\mathbf{x})=1+\sum_{i=1}^{d}\mathbf{x}(i)C^{i-1}$$ is a bijection.

Let $\mathbf{z}$ be an element chosen uniformly at random from the set $B$, and let $r=(\mathbb{E}[|\mathbf{z}|^{2}])^{1/2}$. We have
$$r^2=\mathbb{E}[|\mathbf{z}|^{2}]=\sum_{i=1}^{d}\mathbb{E}[\mathbf{z}(i)^{2}]=\frac{d(C-1)(2C-1)}{6}\geq \frac{dC^{2}}{6}.$$

Let $A'$ consist of the set of all points in $B$ that lie in the spherical shell between the spheres of radii $r(1-\epsilon)$ and $r(1+\epsilon)$ about the origin. That is, let $$S=\{\mathbf{x}:r(1-\epsilon)\leq |\mathbf{x}|\leq r(1+\epsilon)\},$$
and let $A'=B\cap S$. Finally, set $A=\phi(A')$. Next we show, using standard concentration laws, that $A'$ contains almost all elements of $B$ and, hence, $A$ contains almost all elements of $[n]$.

\begin{claim}\label{harom}
	$|A|=|A'|\geq C^{d}(1-2e^{-\frac{1}{18}d\epsilon^{2}})=n-2n^{1-\frac{\epsilon^{2}}{18\log C}}$.
\end{claim}

\begin{proof}
	Note that $|\mathbf{z}|^{2}=\sum_{i=1}^{d}\mathbf{z}(i)^{2}$ is the sum of $d$ independent random variables taking values in $\{0,\dots,(C-1)^{2}\}$.  We have $r^{2}=\mathbb{E}[|\mathbf{z}|^{2}]\leq C^{2}d$. On the other hand, if $\mathbf{x}\not\in A'$, then $||\mathbf{x}|^{2}-r^{2}|>\epsilon r^{2}> (1/6)\epsilon dC^{2}$. Thus, by Hoeffding's inequality \cite{H63}, we obtain
	$$1-\frac{|A'|}{C^{d}}\leq \mathbb{P}\left[||\mathbf{z}|^{2}-r^{2}|> (1/6)\epsilon dC^{2}\right]\leq 2e^{-\frac{1}{18}d\epsilon^{2}}=2n^{1-\frac{\epsilon^{2}}{18\log C}}.$$
\end{proof}

Therefore, with the choice $\alpha=1-\frac{\epsilon^{2}}{30\log C}$, we have $|A|\geq n-n^{\alpha}$, provided that $n$ is sufficiently large.
\smallskip

It remains to define a coloring $c$ of $A$ with the desired properties. Using the bijection $\phi$ between $B$ and $[n]$, this corresponds to a coloring of $A'\subset B$. We would like to guarantee that for every $a,b\in A$ with $a\neq b$ and $c(a)=c(b)$, we have $\frac{a+b}{2} \mbox{ and } 2a-b\not\in A$. (By swapping $a$ and $b$, the latter condition also implies that $2b-a\not\in A$.) Equivalently, we want to have
$$\phi^{-1}\left(\frac{a+b}{2}\right)\;\;\mbox{ and}\;\;\; \phi^{-1}(2a-b)\not\in A'.$$
To achieve this, we would like to use the identities $$\phi^{-1}\left(\frac{a+b}{2}\right)=\frac{\phi^{-1}(a)+\phi^{-1}(b)}{2}\;\; \mbox{ and}\;\;\; \phi^{-1}(2a-b)=2\phi^{-1}(a)-\phi^{-1}(b).$$ However, these equations hold if and only if
$$\frac{\phi^{-1}(a)+\phi^{-1}(b)}{2}\in B\;\; \mbox{ and}\;\;\; 2\phi^{-1}(a)-\phi^{-1}(b)\in B,$$ respectively.
\smallskip

To overcome this problem, we first give an auxiliary coloring $f$ of $B$ such that if $f(\mathbf{x})=f(\mathbf{y})$, then
$$\frac{\mathbf{x}+\mathbf{y}}{2}\;\; \mbox{ and}\;\;\;2\mathbf{x}-\mathbf{y}\in B.$$ We define $f$ as follows. For any $\mathbf{x}\in B$, let $f(\mathbf{x})=(a_1,\dots,a_d,b_1,\dots,b_d)$, where, for every $i\in [d]$, we have
$$a_{i}=\begin{cases}
0 &\mbox{if } \mathbf{x}(i)\mbox{ is even},\\
1 &\mbox{if } \mathbf{x}(i)\mbox{ is odd}.
\end{cases}$$	
 and
 $$b_{i}=\begin{cases}
 k &\mbox{if }\, \mathbf{x}(i)\leq \frac{C}{2}\mbox{ and }\, 2^{k-1}-1\leq \mathbf{x}(i)< 2^{k}-1,\\
 -k &\mbox{if }\, \mathbf{x}(i)>\frac{C}{2}\mbox{ and }\, 2^{k-1}-1\leq C-1-\mathbf{x}(i)< 2^{k}-1.
 \end{cases}$$
Then $f$ uses at most $2^d(2\log_2 C)^d$ colors, and it is easy to verify that $f$ satisfies the desired properties.
\smallskip

Next, we color each color class of $f$ with a separate set of colors such that any triple of (collinear) points in $A'$ that corresponds to a $3$-AP in $A$ is rainbow. (It will cause no confusion if such a point triple will also be called a $3$-AP.) In order to achieve this, we need a simple geometric observation; see Figure \ref{figure1}.

\begin{claim}\label{claim:circle}
  If $\mathbf{x},\mathbf{y},\mathbf{z}\in S$ such that $\mathbf{y}=\frac{\mathbf{x}+\mathbf{z}}{2}$, then $|\mathbf{x}-\mathbf{z}|\leq 4\sqrt{\epsilon}r.$
\end{claim}

\begin{proof}
  At least one of the angles $\mathbf{0}\mathbf{y}\mathbf{x}$ and $\mathbf{0}\mathbf{y}\mathbf{z}$ is at least $\frac{\pi}{2}$, see Fig.~1. Assume without loss of generality that $\mathbf{0}\mathbf{y}\mathbf{z}$ is such an angle. Then we have $|\mathbf{y}|^2+|\mathbf{y}-\mathbf{x}|^2\leq |\mathbf{x}|^2.$ On the other hand, $|\mathbf{x}|^2\leq (1+\epsilon)^2r^2$ and $|\mathbf{y}|^2\geq (1-\epsilon)^{2}r^2$, so that we obtain $$|\mathbf{x}-\mathbf{z}|^2=4|\mathbf{y}-\mathbf{x}|^2\le 4(|\mathbf{x}|^2-|\mathbf{y}|^2)\leq 16\epsilon r^{2}.$$
\end{proof}

 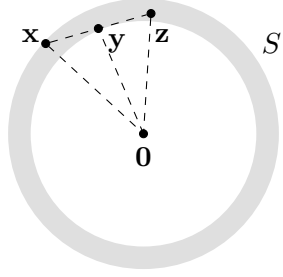
\begin{figure}[t]
	\begin{center}
		\begin{tikzpicture}[scale=1]	
		
		\fill[gray!25,even odd rule] (0,0) circle (1.8) (0,0) circle (1.5); 
		\node at (1.7,1.2) {$S$} ;
		
		
		\node[vertex] at (0,0) {}; \node at (0,-0.3) {$\mathbf{0}$} ;
		
		\node[vertex] (y) at (-0.6,1.4) {} ; \node at (-0.35,1.2) {$\mathbf{y}$} ;
		\node[vertex] (x) at (-1.3,1.2) {} ; \node at (-1.5,1.3) {$\mathbf{x}$} ;
		\node[vertex] (z) at (0.1,1.6) {} ; \node at (0.25,1.3) {$\mathbf{z}$} ;
		\draw[dashed] (x) -- (z) -- (0,0) -- (x) ; \draw[dashed] (0,0) -- (y) ;
		\end{tikzpicture}
		\caption{An illustration for Claim \ref{claim:circle}.}
		\label{figure1}
	\end{center}
\end{figure}

Fix any color class $V$ of $f$, and define a graph $G$ on the vertex set $V\cap A'$, as follows. Join $\mathbf{x},\mathbf{y}\in V\cap A'$ by an edge if at least one of the $3$ vectors $2\mathbf{x}-\mathbf{y},\frac{\mathbf{x}+\mathbf{y}}{2},2\mathbf{y}-\mathbf{x}$ belongs to $A'$.

\begin{claim}\label{otodik}
	Let $\Delta$ denote the maximum degree of the vertices of $G$. Then we have $$\Delta<2^d C^{16\epsilon dC^2}.$$
\end{claim}
\begin{proof}
Fix any $\mathbf{x}\in A'$. By Claim \ref{claim:circle}, every neighbor of $\mathbf{x}$ is at distance at most $4\sqrt{\epsilon}r<4\sqrt{\epsilon} \sqrt{d}C$ from $\mathbf{x}$. If $|\mathbf{x}-\mathbf{y}|\leq 4\sqrt{\epsilon}\sqrt{d}C$ for some $\mathbf{y}\in A'$, then there are at most $16\epsilon d C^2$ indices $i\in [d]$ such that $\mathbf{x}(i)\neq\mathbf{y}(i)$. The number of vertices $\mathbf{y}$ with this property is smaller than $2^{d}C^{16\epsilon dC^2}.$ Indeed, there are fewer than $2^{d}$ ways to choose the indices $i$ for which  $\mathbf{x}(i)\neq\mathbf{y}(i)$ and, for each such index $i$, there are fewer than $C$ different choices for $\mathbf{y}(i)$. Therefore, we have
  $$\Delta< 2^{d}C^{16\epsilon dC^2}.$$
\end{proof}

It follows from Claim~\ref{otodik} that $G$ has a proper coloring with at most $\Delta+1$ colors. By the definition of $G$, if in such a coloring two elements are colored with the same color, then this pair is not contained in any $3$-AP in $A'$.

Repeating the same procedure for each color class $V$ of $f$, using a different set of colors every time, we obtain a coloring $c'$ of $A'$ with at most
$$(\Delta+1)2^d(2\log_2 C)^d\leq(10C^{16\epsilon C^2}\log_2 C)^d$$ colors such that every $3$-AP in $A'$ is rainbow. The coloring $c$ on $A$ induced by $c'$ has the same property.

Using that $\epsilon=\frac{1}{C^3}$, we have $D=10C^{16\epsilon C^2}\log_{2}C<C$, provided that $C$ is sufficiently large. Letting $\beta=\log_{C}D$, the number of colors used by $c$ is at most $n^{\beta}$.

Increasing $C$, $\beta$ tends to zero. Thus, in view of Claim~\ref{harom}, we obtain that for every $\beta>0$, there is a suitable positive $\alpha<1$ which satisfies the conditions of Theorem~\ref{thm:mainthm}.  \hfill$\Box$

\section{Induced matchings---Proof of Corollary \ref{cor:maincor}}

For simplicity, we only prove the corollary when $n$ is a perfect square, that is, $n=m^2$ for some integer $m$. Let $V$ be a set of size $n$, and partition it into $m$ sets $V_1,\dots,V_m$ of size $m$. Let $\alpha,\beta<1$ denote two constants meeting the requirements of  Theorem~\ref{thm:mainthm}. We will show that Corollary \ref{cor:maincor} is true with suitable constants $\alpha'=\frac{\alpha+1}{2}+o(1)$ and $\beta'=\frac{\beta+1}{2}+o(1)$, as $n\rightarrow\infty$.
\smallskip

Let $A\subset [2m]$ be a set of size at least $2m-(2m)^{\alpha}$, and let $c$ be a coloring of $A$ with at most $(2m)^{\beta}$ colors such that every 3-AP in $A$ is rainbow.

Construct a graph $G$ on the vertex set $V$, as follows. Identify each $V_i$ with the set $[m]$ and, for every $1\leq i<j\leq m$ and $x\in V_i, y\in V_j$, connect $x$ and $y$ by an edge of $G$ if and only if $x+y\in A$. If $xy$ is an edge, color it with the color $$c'(xy)=(i,j,x-y,c(x+y)).$$
Note that the same symbol $x$ denotes a different vertex in each $V_i$. Also, the third coordinate of the color $c'(xy)$ can be negative, zero, or positive. 

First, we show that each color class is an induced matching. In other words, we show that if $xy\neq uv$ are distinct edges of $G$ such that $c'(xy)=c'(uv)=c'$, then $xy$ and $uv$ do not share a vertex and none of $xu,xv,yu,yv$ can be an edge of $G$ having color $c'$. The first two coordinates of the color $c'(xy)=c'(uv)=c'$ determine the pair of indices $(i,j), i<j,$ such that both $xy$ and $uv$ run between $V_i$ and $V_j$. Suppose without loss of generality that $x,u\in V_i$ and $y,v\in V_j$. If $x=u$, say, then $c'(xy)=c'(uv)$ implies that $x-y=u-v$, so that $y=v$, contradicting our assumption that $xy$ and $uv$ are distinct edges. Therefore, $xy$ and $uv$ cannot share a vertex. By definition, there is no edge between $x$ and $u$, and there is no edge between $y$ and $v$.

It remains to show that neither $xv$, nor $yu$ can be an edge of color $c'$. Let $d=x-y=u-v$. Suppose, for example, that $xv$ is an edge of color $c'$. Then $x+v\in A$, and we have $$\frac{(x+y)+(u+v)}{2}=\frac{(2x-d)+(2v+d)}{2}=x+v.$$
Comparing the left-hand side and the right-hand side, it follows that $x+y,x+v,u+v$ are distinct numbers that form a 3-AP in $A$. However, the fourth coordinate of the color $c'(xy)=c'(uv)=c'$ guarantees that $c(x+y)=c(u+v)$. Thus, we have found a non-rainbow 3-AP in $A$, contradicting our assumptions. A symmetric argument shows that $yu$ cannot be an edge of color $c'$ either.
\smallskip

Let us count the number of edges of $G$. For every pair $(i,j), 1\leq i<j\leq m$, there are at least $m^2-m(2m)^{\alpha}>m^2-2m^{1+\alpha}$ edges between $V_i$ and $V_j$. Indeed, for every $s\in [2m]\setminus A$, there are at most $m$ pairs $(x,y)\in [m]^2$ such that $x+y=s$, and the number of such elements $s$ is at most $(2m)^\alpha$. Hence, we have $$|E(G)|\geq \binom{m}{2}(m^2-2m^{1+\alpha})\geq\binom{n}{2}-2n^{3/2+\alpha/2}.$$

The number of colors used by $c'$ and, therefore, the number of induced matchings $G$ can be partitioned into, is at most $m^2(2m)(2m)^{\beta}\leq 4n^{3/2+\beta/2}$. This completes the proof of Corollary \ref{cor:maincor}.    \hfill$\Box$

\section{Concluding remarks}

Our proof of Theorem \ref{thm:mainthm} can be easily extended to longer arithmetic progressions.

\begin{theorem}\label{general}
	For any positive integer $k$, there exist $\alpha,\beta>0$ with the following property. For every sufficiently large positive integer $n$, there are a set $A\subset [n]$ with $|A|\geq n-n^{\alpha}$ and a coloring of $A$ with at most $n^{\beta}$ colors such that every  arithmetic progression of length at most $k$ in $A$ is rainbow.

Moreover, for every $\beta > 0$, we can choose $\alpha < 1$ satisfying the above conditions.
\end{theorem}

In order to establish Theorem~\ref{general}, we need to modify the proof of Theorem \ref{thm:mainthm} at the following two points.

\begin{enumerate}
	\item We should construct an auxiliary coloring $f$ on $B$ such that if $f(\mathbf{x})=f(\mathbf{y})$, then $\frac{p}{q}\mathbf{x}+(1-\frac{p}{q})\mathbf{y}\in B$ for every $p,q\in [k]$. Color $(x_1,\dots,x_d)$ with the color $(a_1,\dots,a_d,b_1,\dots,b_d)$, where $a_{i}\in \{0,\dots,k!-1\}$ such that $a_i\equiv x_i \mod k!$, and 
	 $$b_{i}=\begin{cases}
	j &\mbox{if }\, \mathbf{x}(i)\leq \frac{C}{2}\mbox{ and }\, (\frac{k}{k-1})^{j-1}-1\leq \mathbf{x}(i)<(\frac{k}{k-1})^{j}-1,\\
	-j &\mbox{if }\, \mathbf{x}(i)>\frac{C}{2}\mbox{ and }\, (\frac{k}{k-1})^{j-1}-1\leq C-1-\mathbf{x}(i)< (\frac{k}{k-1})^{j}-1.
	\end{cases}$$
	Then $f$ uses $(k^{k}\log C)^{O(d)}$ colors.
	
	\item  Instead of Claim \ref{claim:circle}, we can show that if $\mathbf{x}_1,\dots,\mathbf{x}_k$ is a $k$-term arithmetic progression in $S$, then $|\mathbf{x}_1-\mathbf{x}_k|\leq 10 \sqrt{\epsilon} r$.
\end{enumerate}

After these changes, the proof can be completed by straightforward calculations, in the same way as in the case of Theorem~\ref{thm:mainthm}.

\section{Acknowledgments}

We would like to thank Benny Sudakov for valuable discussions.


\begin{thebibliography}{99}
    \bibitem{A02}
	N. Alon.
	''Testing subgraphs in large graphs.''
	Proc. 42nd IEEE Annual Symposium of Foundations of Computer Science (FOCS), IEEE, 434--441, 2001. Also: Random Structures and Algorithms 21, 359--370, 2002.
	
	\bibitem{AMS13}
	N. Alon, A. Moitra, and B. Sudakov.
	''Nearly complete graphs decomposable into large induced matchings and their applications.''
	Journal of European Mathematical Society 15, 1575--1596, 2013.
	
	\bibitem{AS03}
	N. Alon and A. Shapira.
	''Testing subgraphs in directed graphs.''
	Proc. 35th ACM Symposium on Theory of Computing (STOC), ACM Press, 700--709, 2003. Also: JCSS 69, 354--382, 2004.
	
	\bibitem{AS04}
	N. Alon and A. Shapira.
	''A characterization of easily testable induced subgraphs.''
	Proc. of the Fifteenth Annual ACM-SIAM SODA, 935--944, 2004. Also: Combinatorics, Probability and Computing 15, 791--805, 2006.
	
	\bibitem{B46}
	F. A. Behrend.
	''On sets of integers which contain no three terms in arithmetic progression.''
	Proc. National Academy of Sciences USA 32, 331--332, 1946.
	
	\bibitem{BLM}
	Y. Birk, N. Linial, and R. Meshulam.
	''On the uniform-traffic capacity of single-hop inter-connections employing shared directional multichannels.''
	IEEE Transactions on Information Theory, 186--191, 1993.
	
	\bibitem{FNRRS02}
	E. Fischer, I. Newman, S. Raskhodnikova, R. Rubinfeld, and A. Samorodnitsky. ''Monotonicity testing over general poset domains.''
	Proc. 34th ACM Symposium on Theory of Computing (STOC), 474--483, 2002.
	
	\bibitem{FHS17}
	J. Fox, H. Huang, and B. Sudakov.
	''On graphs decomposable into induced matchings of linear sizes.''
	Bulletin of the London Mathematical Society 49 (1), 45--57, 2017.
	
	\bibitem{HW03}
	J. H\aa stad and A. Wigderson.
	''Simple analysis of graph tests for linearity and PCP.''
	Random Structures and Algorithms, 139--160, 2003.

\bibitem{H63}
W. Hoeffding. ''Probability inequalities for sums of bounded random variables.'' Journal of the American Statistical Association 58 (301), 13--30, 1963.
		
\bibitem{R52}
K. Roth. ''Sur quelques ensembles d'entiers (in French).''
C. R. Acad. Sci. Paris 234, 388--390, 1952.

    \bibitem{RSz78}
	I. Ruzsa and E. Szemer\'edi.
	''Triple systems with no six points carrying three triangles.''
	Colloquia Mathematica Societatis J\'anos Bolyai, 939--945, 1978.
	
\bibitem{Sz75}
E. Szemer\'edi. ''On sets of integers containing no $k$ elements in arithmetic progression.'' Acta Arith. 27, 199--245, 1975.

    \bibitem{Sz78}
	E. Szemer\'edi.
	''Regular partitions of graphs.''
	 In: Proc. Colloque Inter. CNRS (J. C. Bermond,J. C. Fournier, M. Las Vergnas and D. Sotteau, eds.), 399--401, 1978.

\bibitem{vdW27}
B. L. van der Waerden. ''Beweis einer Baudetschen Vermutung.'' Nieuw. Arch. Wisk. (in German). 15, 212--216, 1927.

\end{thebibliography}
\end{document}